\documentclass[12pt]{amsart}
\usepackage{amsfonts}
\usepackage{amsfonts}
\usepackage{amssymb}
\usepackage{nicefrac}
\usepackage{enumerate}
\usepackage{url}
\usepackage{color}      
\usepackage{epsfig}
\usepackage{graphicx}
\usepackage{eufrak}
\usepackage{tikz}

\newcommand{\Chi}{\mbox{\Large$\chi$}_I}
\newcommand{\CChi}{\mbox{\Large$\chi$}}

\newcommand{\vv}{\mathfrak{v}_{(f,I,c)}}
\newcommand{\R}{\mathbb{R}}
\newcommand{\N}{\mathbb{N}}

\newcommand{\T}{\mathbb{T}_m}
\renewcommand{\S}{\mathcal{S}}

\newcommand{\F}{\mathcal{F}}

\newcommand{\defi}{\mathrel{\mathop :}=}
\DeclareMathOperator{\med}{\,median\,}
\newtheorem{teo}{Theorem}[section]
\newtheorem{lem}[teo]{Lemma}
\newtheorem{co}[teo]{Corollary}
\newtheorem{pro}[teo]{Proposition}

\theoremstyle{remark}
\newtheorem{re}[teo]{Remark}
\newtheorem{ex}[teo]{Example}

\theoremstyle{definition}
\newtheorem{de}[teo]{Definition}

\title[Nonlinear Harmonic Measures on Trees]{Estimates for
Nonlinear Harmonic Measures on Trees}
\author[L. M. Del Pezzo, C. A. Mosquera and J. D. Rossi]
{Leandro M. Del Pezzo, Carolina A. Mosquera and Julio D. Rossi}

\address{Leandro M. Del Pezzo and Carolina A. Mosquera
\hfill\break\indent
CONICET and Departamento  de Matem{\'a}tica, FCEyN,
Universidad de Buenos Aires,
\hfill\break\indent Pabellon I, Ciudad Universitaria (1428),
Buenos Aires, Argentina.}

\email{{\tt ldpezzo@dm.uba.ar, mosquera@dm.uba.ar}}

\address{Julio D. Rossi \hfill\break\indent
Departamento  de An{\'a}lisis Matem{\'a}tico, Universidad de Alicante,
\hfill\break\indent Ap. correo 99, 03080, Alicante, SPAIN.
}

\email{{\tt julio.rossi@ua.es}}

\thanks{
Leandro M. Del Pezzo was partially supported by UBACyT 20020110300067
and CONICET PIP 5478/1438  (Argentina) ,
Carolina A. Mosquera was partially supported by UBACyT 20020100100638, PICT 0436 and
CONICET PIP  112 200201 00398 (Argentina)
and Julio D. Rossi  was partially supported by MTM2011-27998,
(Spain) }

\begin{document}
\begin{abstract}
In this paper we give some estimates for
nonlinear harmonic measures on trees. In particular, we
estimate in terms of the size of a set $D$ the value at the
origin of the solution to
$
u(x)=F((x,0),\dots,(x,m-1))$ for every $x\in\T,
$ a directed tree with $m$ branches
with initial datum $f+\chi_D$.
Here $F$ is an averaging operator on $\R^m$,
$x$ is a vertex of a  directed tree $\T$ with regular
$m$-branching and $(x,i)$ denotes a successor of that vertex
for $0\le i\le m-1$.
\end{abstract}

\maketitle

\section{Introduction}
Let us first recall some well known facts for the classical $p-$Laplacian.
Let $\Omega$ be the unit ball in $\R^N,$ $N>1.$  We say that $u$ is
$p-$harmonic ($p-$superharmonic/$p-$subharmonic) in $\Omega$ ($1<p<\infty$) if
$u\in W^{1,p}(\Omega)$ and
\[
\int_{\Omega}|\nabla u|^{p-2}\nabla u\nabla\psi\, dx=0, \qquad (\geq 0 /\leq 0)
\]
whenever $\psi \in C^{\infty}_0(\Omega)$ ($\psi \geq 0$).
Let $E$ be a subset of $\partial\Omega.$ Consider the following class
\[
U_p(E)=\Big\{v\colon v\ge 0 \mbox{ and } p-\mbox{superharmonic s. t.}\liminf_{x\in\Omega\,x\to y}v(y) \ge\CChi_E(y)\, \forall
y\in\partial\Omega\Big\}.
\]
The $p-$harmonic measure of the set $E$ relative to the
domain $\Omega$ is the function $\omega_p(\cdot, E)$
whose value at any $x\in\Omega$ is given
by
$
\omega_p (x,E) = \inf\{v(x) \colon v \in U_p(E)\}$.
We simply denote $\omega_p(E)$ when $x=0.$
For a deeper discussion about of $p-$harmonic measure, we refer the reader to \cite{AM, BBS, GLM, HKM, K, LMW, M1}.

In this context, the following problem for the $p-$Laplacian
remains open, see \cite{wolf}.

\noindent {\bf Boundary Comparison Principle.} For $\delta>0$
consider $I_\delta$ a spherical cap with length
$\nicefrac{\delta}2.$ Given $\epsilon>0,$ find
$\delta=\delta(\epsilon,M,p)>0$ such that
\[
|u(0)-v(0)|<\epsilon
\]
for all $p-$harmonic functions $u$ and $v$ in $\Omega$ that extend
to $\overline{\Omega},$ are bounded $\|u\|_{\infty}\le M,$
$\|v\|_{\infty}\le M,$ and satisfy $u(y)=v(y)$ for all
$y\in\partial\Omega\setminus I_\delta.$

Closely related to this problem is the following:

\noindent{\bf $p-$harmonic Measure Estimates.} Does there exist
$\alpha>0$ such that
\[
\omega_p(I_\delta)\sim\delta^{\alpha} \qquad \mbox{ as }\delta \to 0 ?
\]

In \cite{PSSW}, the authors study the second question  in the case $p=\infty.$ They showed that
$\omega_{\infty}(I_\delta)\sim\delta^{\frac{1}{3}}$. Similar ideas can be used to obtain the result for $1<p<\infty$, see \cite{PS}.

In this work
we provide answers to both problems for the $F$-harmonic function
in a directed tree where  $F$ is an averaging operator on $\R^m,$
see below for a precise definition. Regular trees are discrete models
of the unit ball of $\R^N$ and hence our results can be seen as a contribution
in order to the study of the previously mentioned open problem.

We remark that for the linear case, $p=2$, the solutions to these problems are well known and the starting point for their study is the mean value property for harmonic functions. One of the main interests of the present work is to show what kind of results
can be proved when the mean value property under consideration is nonlinear.

Now, let us introduce briefly some definitions and notations needed to make
precise the statements of our main results (but we refer the reader to Section~2 where more details can be found).
Let $F\colon\R^m\to\R$ be a continuous function. We call $F$
an averaging operator if it satisfies the following: $$F(0,\dots,0)=0 \mbox{ and } F(1,\dots,1)=1;$$
$$F(tx_1,\dots,tx_m)=t F(x_1,\dots,x_m);$$
$$F(t+x_1,\dots,t+x_m)=t+ F(x_1,\dots,x_m),$$
for all $t\in\R;$ $$F(x_1,\dots,x_m)<\max\{x_1,\dots,x_m\},$$ if not all
		   $x_j$'s are equal; $F$ is nondecreasing with respect to each variable; in addition, we will assume that
$F$ is permutation
invariant, that is,
$
F(x_1,\dots,x_m)= F(x_{\tau(1)},\dots,x_{\tau(m)})
$
for each permutation $\tau$ of $\{1,\dots,m\}$ and that there exists $0<\kappa<1$
such that
\begin{equation}\label{Pro.intro}
F(x_1+c,\dots,x_m)\le F(x_1,\dots,x_m)+c\kappa
\end{equation}
for all $(x_1,\dots,x_m)\in\R^m$ and for all $c>0.$

As examples of averaging operators we mention the following:
The first example is taken from \cite{KLW}. For $1<p<+\infty,$
	the operator
	$F^p(x_1,\dots,x_m)=t$ from $\R^m$ to $\R$  defined
	implicity by
	\[
	\sum_{j=1}^m (x_j-t)|x_j-t|^{p-2}=0
	\]
	is a permutation invariant averaging operator.
Next, we consider, for $0\le\alpha\le1$ and $0<\beta \leq 1$ with
		$\alpha+\beta=1$
		\begin{align*}
			&F_0(x_1,\dots,x_m)=\frac{\alpha}2
			\left(\max_{1\le j\le m}
			\{x_j\}+\min_{1\le j\le m}\{x_j\}
			\right) + \frac{\beta}m\sum_{j=1}^m x_j,\\
			&F_1(x_1,\dots,x_m)=\alpha
			\underset{{1\le j\le m}}{\med}\{x_j\}+
			 \frac{\beta}m\sum_{j=1}^m x_j,\\
			 &F_2(x_1,\dots,x_m)=\alpha
			\underset{{1\le j\le m}}{\med}\{x_j\}+
			 \frac{\beta}2
			 \left(\max_{1\le j\le m}
			\{x_j\}+\min_{1\le j\le m}\{x_j\}
			\right),
		\end{align*}
		where
		\[
\underset{{1\le j\le m}}{\med}\{x_j\}\, \defi
   \begin{cases}
  y_{\frac{m+1}2} & \text{ if }m \text{ is even},  \\
  \displaystyle\frac{y_{\frac{m}2}+ y_{(\frac{m}2 +1)}}2
  & \text{ if }m \text{ is odd},
  \end{cases}
\]
with $\{y_1,\dots, y_m\}$ a nondecreasing rearrangement of
$\{x_1,\dots, x_m\}.$
$F_0, F_1$ and $F_2$ are permutation invariant
averaging operators. Note that $F_0$ and $F_1$ verify \eqref{Pro.intro} but $F_2$ does not.

Associated with an averaging operator $F$ we have an equation on a tree.
In what follows $x$ is a vertex of a directed tree $\T$ with regular
$m$-branching and $(x,i)$ is a successor of that vertex
for all $0\le i\le m-1$ (see Section 2 for more details).
A function $u\colon\T\to\R$ is called $F$-subharmonic
function if the
inequality
$
u(x)\le F(u(x,0),\dots,u(x,m-1))
$
holds for all $x\in\T$ and $F$-superharmonic if the opposite
inequality holds for all $x\in\T.$ We say that $u$ is $F$-harmonic
if $u$ is  both $F$-subharmonic and $F$-superharmonic, that is, $u$ is a solution to the equation
\[
u(x) = F(u(x,0),\dots,u(x,m-1)).
\]

Let $f:[0,1]\to\R$ be a bounded function,
$c>0$ and $E\subset\partial\T.$
We define $U_F(f,E,c)$ as the set of all $F-$superharmonic functions $u$ such that
\[
\liminf_{k\to+\infty} u(x_k)\ge
f(\pi)+c\CChi_{E}(\pi),
\]
for all $\pi=(x_1, \dots, x_k, \dots)\in\partial\T.$

Finally, we need the notion of solution to the Dirichlet Problem, (DP) in the sequel. Given  a bounded function $f\colon [0,1]\to\R,$ $u$ is a solution to the Dirichlet Problem with boundary data $f$ if it is $F$-harmonic
and verifies
\[
\lim_{k\to+\infty} u(x_k)= f(\pi), \qquad \forall \pi=(x_1, \dots, x_k, \dots)\in\partial\T.
\]

Now, we are ready to state the main result of this paper.

\begin{teo}\label{teoestimacion}
	Let $F$ be a permutation invariant averaging operator with
	the property \eqref{Pro.intro}, $f:[0, 1]\to\R$ be a continuous
	function, $I$ be a subinterval of $[0,1]$ with measure $|I|$ and $c>0.$
	If $u$ is the solution of \textnormal{(DP)}
	with boundary
	data $f$, then
	\[
		0\le \inf\Big\{w(\emptyset)-u(\emptyset)
		\colon w\in U_F(f,I,c)\Big\}\le
		2c\left(m|I|\right)^\gamma
	\]
	for any $\gamma\le -
	\log_m(\kappa).$
\end{teo}

Note that the obtained bound depends precisely both on the tree ($m$ is the parameter that controls the branching of the tree), the interval where the perturbation of $f$ takes place (through its measure), the size of the perturbation (given by $c$) and on the operator $F$ ($\kappa$ appears in \eqref{Pro.intro}).

We immediately deduce, taking $f\equiv 0$ and $c=1,$ the following
result.

\begin{co}[$p-$harmonious Measure Estimates]
	If $I$ is a subinterval of $[0,1],$ then
	\begin{equation*}
		0
	\le \omega_F(I)\le	
	2 \left(m|I|\right)^\gamma
	\end{equation*}
	for all $\gamma\le-\log_m(\kappa).$
\end{co}

Finally, we state the Boundary Comparison Principle.

\begin{co}[Boundary Comparison Principle]\label{BCP} Let $F$
be a permutation invariant averaging operator that satisfies
the property \eqref{Pro.intro}. Given $\varepsilon>0$ and $M>0,$
there exists $\delta=\delta(\varepsilon,M, m)>0$ so
that, if
$f,g\colon [0,1]\to\mathbb{R}$ are continuous functions,
$f=g$ in $[0,1]\setminus I,$ and
$\|f\|_\infty+\|g\|_\infty\le M,$
where $I$ is a subinterval of $[0,1]$ with
$|I|=\delta,$ then
\[
|u(\emptyset)-v(\emptyset)|<\varepsilon
\]
where $u$ and $v$ are the solutions of
\textnormal{(DP)} with boundary data
$f$ y $g $ respectively
\end{co}

Let us end the introduction with a brief comment on previous bibliography.
For nonlinear mean values on a finite graph we refer to \cite{Ober} and references therein. For equations on trees like the ones considered here, see \cite{ary,KLW,KW} and \cite{s-tree,s-tree1}, where it is proved the existence and uniqueness of a solution using game theory. Here we use ideas from these references.
Nonlinear mean value properties that characterize solutions to PDEs can be found, for example, in \cite{MPR}, \cite{rv},
\cite{hr} and \cite{hr1}. These mean value properties reveal to be quite useful when designing numerical schemes that approximate solutions to the corresponding nonlinear PDEs, see \cite{O1,O2}.

{\bf Organization of the paper.}  In Section
\ref{previa} we collect some preliminary facts concerning
trees, averaging operators, $F-$harmonic functions and
$F-$harmonic measures; in Section \ref{DP} we prove existence and
uniqueness for the Dirichlet problem and a comparison principle;
Finally in Section \ref{esti} we prove Theorem \ref{teoestimacion} and Corollary \ref{BCP}.

\section{Preliminaries}\label{previa}

We begin with a review of the basic results that will be needed in subsequent sections.
The known results are generally stated without proofs, but we provide references where
the proofs can be found. Also, we introduce some of our notational conventions.

\subsection{Directed Tree}
Let $m\in\mathbb{N}$ be at least 3. In this work we consider a directed
tree $\T$ with regular $m-$branching, that is, $\T$ consists of
the empty set $\emptyset$ and all finite  sequences
$(a_1,a_2,\dots,a_k)$ with $k\in\N,$ whose coordinates $a_i$ are
chosen from $\{0,1,\dots,m-1\}.$ The elements in $\T$ are called
vertices. Each vertex $x$ has $m$ successors, obtained by adding
another coordinate. We will
denote by $\S(x)$ the set of successors of the vertex $x.$ A
vertex $x\in\T$ is called an $n-$level vertex ($n\in\mathbb{N}$) if
$x=(a_1,a_2,\dots,a_n).$  The set of all $n-$level vertices is
denoted by $\T^n.$

\begin{ex}
  Let $\kappa\in\mathbb{N}$ be at least 3.
  A $\nicefrac{1}{\kappa}-$Cantor set, that we denote by
  $C_{\nicefrac{1}{\kappa}}$,
is the set of all $x\in[0,1]$ that have a base $\kappa$ expansion
without the digit $1$,
  that is $x=\sum a_j\kappa^{-j}$ with
  $a_j\in\{0,1,\dots,\kappa-1\}$ with $a_j \neq 1$. Thus $C_{\nicefrac{1}{\kappa}}$
  is obtained from $[0,1]$ by removing the second $\kappa-$th part of the line segment $[0,
  1]$, and then removing the second interval of length $\nicefrac1{\kappa}$
  from the remaining intervals, and so on. This set
  can be thought of as a directed tree with regular $m-$branching with $m=\kappa-1$.

  For example, if $\kappa=3$, we identify $[0, 1]$ with $\emptyset,$
  the sequence $(\emptyset, 0)$ with the first interval right $[0,
  \nicefrac13]$,
  the sequence $(\emptyset, 1)$ with the central interval $[\nicefrac13,
  \nicefrac23]$ (that is removed),
  the sequence $(\emptyset, 2)$ with the left interval  $[\nicefrac23, 1],$
  the sequence $(\emptyset, 0 ,0)$ with the interval
  $[0, \nicefrac{1}{9}]$ and so on.
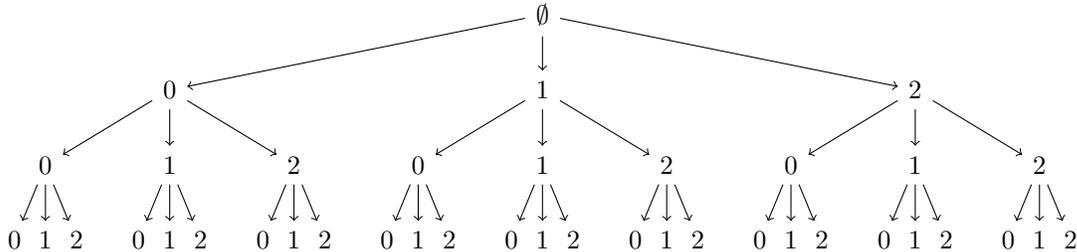
\begin{center}
\begin{tikzpicture} [font=\footnotesize,
grow=down, 
level 1/.style={->, sibling distance=12em},
level 2/.style={->,sibling distance=4em}, level distance=1cm,
level 3/.style={->,sibling distance=1em}, level distance=1cm]

    \node {$\emptyset$}
        child { node {0}
						child { node{0}
										child { node{0}}
										child { node{1}}
                        					child { node{2}}
								}
						child { node{1}
										child { node{0}}
										child { node{1}}
                        					child { node{2}}
								}
                        					child { node{2}
										child { node{0}}
										child { node{1}}
                        					child { node{2}}
								}
				}
       child { node {1}
						child { node{0}
										child { node{0}}
										child { node{1}}
                        					child { node{2}}
								}
						child { node{1}
										child { node{0}}
										child { node{1}}
                        					child { node{2}}
								}
                        					child { node{2}
										child { node{0}}
										child { node{1}}
                        					child { node{2}}
								}
				}
 child { node {2}
						child { node{0}
										child { node{0}}
										child { node{1}}
                        					child { node{2}}
								}
						child { node{1}
										child { node{0}}
										child { node{1}}
                        					child { node{2}}
								}
                        					child { node{2}
										child { node{0}}
										child { node{1}}
                        					child { node{2}}
								}
				}
    ;
\end{tikzpicture}
\end{center}
\end{ex}

\medskip

A branch of $\T$ is an infinite sequence of vertices, each followed by its immediate successor.
The collection of all branches forms the boundary of $\T$,  denoted by $\partial\T$.

\medskip

We now define a metric on $\T\cup \partial\T.$ The distance
between two sequences (finite or infinite) $\pi=(a_1,\dots,
a_k,\dots)$ and $\pi'=(a_1',\dots, a_k',\dots)$ is $m^{-K+1}$ when
$K$ is the first index $k$ such that $a_k\neq a_k';$ but when
$\pi=(a_1,\dots, a_K)$ and $\pi'=(a_1,\dots, a_K,
a_{K+1}',\dots),$ the distance is $m^{-K}.$  Hausdorff measure and
Hausdorff dimension are defined using this metric. We have
that $\T$ and $\partial\T$ have diameter one and $\partial\T$ has
Hausdorff dimension one. Now, we observe that the mapping
$\psi:\partial\T\to[0,1]$ defined as
\[
\psi(\pi)\defi\sum_{k=1}^{+\infty} \frac{a_k}{m^{k}}
\]
is surjective, where $\pi=(a_1,\dots, a_k,\dots)\in\partial\T$ and
$a_k\in\{0,1,\dots,m-1\}$ for all $k\in\mathbb{N}.$ Whenever
$x=(a_1,\dots,a_k)$ is a vertex, we set
\[
 \psi(x)\defi\psi(a_1,\dots,a_k,0,\dots,0,\dots).
\]
We can also associate to a vertex $x=(a_1,\dots,a_k)$ an
interval $I_x$ of length $\frac{1}{m^k}$ as follows
\[
 I_x\defi\left[\psi(x),\psi(x)+\frac1{m^k}\right].
\]
Observe that for all $x\in \T$, $I_x \cap \partial\T$ is the
subset of $\partial\T$ consisting of all branches that start at
$x$.
With an abuse of notation, we will write $\pi=(x_1,\dots,x_k,\dots)$
instead of $\pi=(a_1,\dots,a_k,\dots)$ where $x_1=a_1$ and
$x_k=(a_1,\dots,a_k)\in\S(x_{k-1})$ for all $k\in\N_{\ge2}.$

Finally we will denote by $\T^x$ the set of the vertices $y\in\T$
such that $I_y\subset I_x.$

\subsection{Averaging Operator}
The following definition is taken from \cite{ary}.
Let $F\colon\R^m\to\R$ be a continuous function. We call $F$
an averaging operator if it satisfies the following set of conditions:
\begin{enumerate}[(i)]
	\item $F(0,\dots,0)=0$ and $F(1,\dots,1)=1$;
	\item $F(tx_1,\dots,tx_m)=t F(x_1,\dots,x_m)$ for all $t\in\R;$
	\item $F(t+x_1,\dots,t+x_m)=t+ F(x_1,\dots,x_m)$ for all
			$t\in\R;$
	\item $F(x_1,\dots,x_m)<\max\{x_1,\dots,x_m\}$ if not all
		   $x_j$'s are equal;
    \item $F$ is nondecreasing with respect to each variable.
\end{enumerate}

\begin{re}\label{Fmax}
It holds that, if $(x_1,\dots,x_m),
(y_1,\dots,y_m)\in\R^m,$ then
\[
x_j\le y_j + \max_{1\le j\le m}\left\{ x_j-y_j\right\}
\]
for all $j\in\{1,\dots,m\}$. Let $F$ be an averaging operator. Then, by (iii) and (v),
\[
F(x_1,\dots,x_m)\le F(y_1,\dots,y_m)+
\max_{1\le j\le m}\left\{ x_j-y_j\right\}.
\]
Therefore
\[
F(x_1,\dots,x_m)-F(y_1,\dots,y_m)\le
\max_{1\le j\le m}\left\{ x_j-y_j\right\}.
\]
\end{re}

\begin{re}\label{re-UCP}
If $F$ is an averaging operator then, using (ii) and (iii),
\begin{align*}
F(a,\dots, a, b)&= F(a,\dots,a,a+(b-a))\\
&= a+F(0,\dots,0,b-a)\\
&= a+ (b-a)F(0,\dots,0,1)\\
&=a(1-F(0,\dots,0,1))+bF(0,\dots,0,1)
\end{align*}
for all $a,b\in\R.$
\end{re}

For the proof of the following proposition see \cite{KLW}.

\begin{pro}
	If $F$ is an averaging operator then
	\begin{enumerate}
	\item $F(1-x_1,\dots,1-x_m)=1-F(x_1,\dots,x_m)$
			for all $(x_1,\dots,x_m)\in\R^m;$
	\item There exists $b>0$ such that whenever
	$F(x_1,\dots,x_m)\ge0$ and $\max\{x_1,\dots,x_m\}\le1,$ then
	$\min\{x_1,\dots,x_m\}\ge-b.$
\end{enumerate}
\end{pro}

In Section \ref{esti} we require,
in addition, for $F$ to be permutation
invariant, that is,
\[
F(x_1,\dots,x_m)= F(x_{\tau(1)},\dots,x_{\tau(m)})
\]
for any permutation $\tau$ of $\{1,\dots,m\}.$

\begin{re}
If $F$ is a permutation invariant averaging operator then we have
that
\[
F(1,0,\dots,0, -1)=F(-1,0,\dots,0,1)=-F(1,0,\dots,0,-1).
\]
Therefore $F(1,0,\dots,0, -1)=0.$
\end{re}

In Section \ref{esti}, we will also need the following assumption:
$F$ is a permutation invariant averaging operator
with the property that there exists $0<\kappa<1$
such that
\begin{equation}\label{Pro}
F(x_1+c,\dots,x_m)\le F(x_1,\dots,x_m)+c\kappa,
\end{equation}
for all $(x_1,\dots,x_m)\in\R^m$ and for all $c>0.$

\begin{re}\label{doblePro} If $F$
is a permutation invariant averaging operator
with the property \eqref{Pro}, then
$$
\begin{array}{l}
    F(x_1+c,x_2+c,x_3,\dots,x_m)\le
	F(x_1,x_2+c,x_3,\dots,x_m) + c\kappa\\
	\qquad =F(x_2+c,x_1,x_3,\dots,x_m) + c\kappa
\le F(x_2,x_1,x_3,\dots,x_m) + 2c\kappa\\
	\qquad = F(x_1,\dots,x_m) + 2c\kappa
\end{array}
$$
for all $(x_1,\dots,x_m)\in\R^m$ and for all $c>0.$
In fact by iterating this argument and using Property (iii), we get
$1<  m \kappa $.
\end{re}

Now we give some examples.

\begin{ex}
This example is taken from \cite{KLW}. For $1<p<+\infty,$
	the operator
	$$F^p(x_1,\dots,x_m)=t$$ from $\R^m$ to $\R$  defined
	implicity by
	\[
	\sum_{j=1}^m (x_j-t)|x_j-t|^{p-2}=0
	\]
	is a permutation invariant averaging operator.
\end{ex}

\begin{ex}
For $0\le\alpha\le1$ and $0<\beta\le1$ with
		$\alpha+\beta=1$, let us consider
		\begin{align*}
			&F_0(x_1,\dots,x_m)=\frac{\alpha}2
			\left(\max_{1\le j\le m}
			\{x_j\}+\min_{1\le j\le m}\{x_j\}
			\right) + \frac{\beta}m\sum_{j=1}^m x_j,\\
			&F_1(x_1,\dots,x_m)=\alpha
			\underset{{1\le j\le m}}{\med}\{x_j\}+
			 \frac{\beta}m\sum_{j=1}^m x_j,\\
			 &F_2(x_1,\dots,x_m)=\alpha
			\underset{{1\le j\le m}}{\med}\{x_j\}+
			 \frac{\beta}2
			 \left(\max_{1\le j\le m}
			\{x_j\}+\min_{1\le j\le m}\{x_j\}
			\right),
		\end{align*}
		where
		\[
\underset{{1\le j\le m}}{\med}\{x_j\}\, \defi
   \begin{cases}
  y_{\frac{m+1}2} & \text{ if }m \text{ is even},  \\
  \displaystyle\frac{y_{\frac{m}2}+ y_{(\frac{m}2 +1)}}2
  & \text{ if }m \text{ is odd},
  \end{cases}
\]
with $\{y_1,\dots, y_m\}$ a nondecreasing rearrangement of
$\{x_1,\dots, x_m\}.$

It holds that $F_0, F_1$ and $F_2$ are permutation invariant
averaging operators. Moreover,
$F_0,$ $F_1$ and $F_2$ satisfy
\eqref{Pro} with $\kappa_0=\nicefrac{\alpha}2
+\nicefrac{\beta}m,$
$\kappa_1=\alpha+\nicefrac{\beta}{m}$
and $\kappa_2=\alpha+\nicefrac{\beta}2,$ due to the fact that
for any $c>0$
\begin{align*}
	\max Y+\min Y
	&\le \max X+\min X
	+ c,\\
	\med Y &\le\med X
	+ c,
\end{align*}
where $Y=\{x_1+c,x_2,\dots,x_m\}$ and
$X=\{x_1,x_2,\dots,x_m\}.$
\end{ex}

\subsection{$F$-harmonic Functions} In this subsection
we will present the definition and some properties of $F$-harmonic
functions.

Let $F$ be an averaging operator.
A function $u\colon\T\to\R$ is called $F$-subharmonic
function if the
inequality
\[
u(x)\le F(u(x,0),\dots,u(x,m-1))
\]
holds for all $x\in\T$ and $F$-superharmonic if the opposite
inequality holds for all $x\in\T.$ We say that $u$ is $F$-harmonic
if $u$ is  both $F$-subharmonic and $F$-superharmonic.

\begin{ex} \label{ex.p.harmonic}
For $1<p<+\infty,$ a function $u\colon\T\to\R$ is $p$-harmonic
if
\[
\sum_{j=0}^{m-1}(u(x,j)-u(x))|u(x,j)-u(x)|^{p-2}=0,
\quad\forall x\in\T
\]
that is
\[
F^p(u(x,0),\dots,u(x,m-1))=u(x),\quad \forall x\in\T.
\]
Thus the $p$-harmonic functions and $F^p$-harmonic functions are the
same.
\end{ex}

\begin{ex} \label{ex.p.harmonious}
A function $u\colon\T\to\R$ is called $(\alpha,\beta)$-harmonious
if
\[
u(x) = \frac{\alpha}2
			\left(\max_{1\le j\le m}
			\{u(x,j)\}+\min_{1\le j\le m}\{u(x,j)\}
			\right) + \frac{\beta}m\sum_{j=1}^m u(x,j),
\quad\forall x\in\T
\]
that is
\[
F_0 (u(x,0),\dots,u(x,m-1))=u(x),\quad \forall x\in\T.
\]
These functions are related to game theory, see \cite{MPR3} for the continuous case and \cite{s-tree,s-tree1}
for trees.
\end{ex}

\begin{re}
Let $F$ be an averaging operator and
 $u$ be a $F$-harmonic function. Then
\begin{enumerate}
\item $au+b$ is a $F$-harmonic function for all $a, b
\in\R;$
\item $u^+=\max\{u,0\}$ and $u^-=\max\{-u,0\}$ are $F$-subharmonic
functions.
\end{enumerate}
\end{re}

Next, we collect some properties.

\begin{lem}Let $F$ be an averaging operator.
If $u$ is a bounded above $F$-subharmonic function and
there exists $x\in\T$ such that $u(x)=\max_{y\in\T} u(y)$ then
$u(y)=u(x)$ for any $y\in\T^x.$
\end{lem}

\begin{proof}
Let $M=u(x)=\max_{y\in\T} u(y)$. We first
observe that it is sufficient to show that $u(y)=M$ for all
$y\in\S(x).$
Since $u$ is a $F$-subharmonic function and $F$ is
an averaging operator, we have that
\[
M=u(x)
\le F(u(x,0),,\dots,u(x,m-1))\le \max_{y\in\S(x)}u(y)\\
\le M.
\]
Then
\[
F(u(x,0),\dots,u(x,m-1))= \max_{y\in\S(x)}u(y)=M.
\]
Therefore, by property (iv), we have that $u(x,i)=M$ for all
$0\le i\le m-1,$ i.e. $u(y)=M$ for all $y\in \S(x).$
\end{proof}

In the same manner, we can prove the following lemma.

\begin{lem}Let $F$ be an averaging operator.
If $u$ is a bounded below $F$-superharmonic function and
there exists $x\in\T$ such that $u(x)=\min_{y\in\T} u(y)$ then
$u(y)=u(x)$ for any $y\in\T^x.$
\end{lem}

If $F$ is an averaging operator then $F$ is a continuous function
and therefore the following result holds.

\begin{lem}\label{lim-unif}
Let $F$ be an averaging operator and $\{u_n\}_{n\in\N}$ be a
sequence of $F$-harmonic functions. If
\[
u(x)=\lim_{n\to+\infty}u_n(x)
\]
for all $x\in\T,$ then $u$ is a $F$-harmonic function.
\end{lem}

\medskip

The Fatou set $\F(u)$ of a function $u$ is the set of the branches
$\pi=(x_1,\dots, x_k, \dots)$ on which
\[
\lim_{k\to+\infty}u(x_k)
\]
 exists and is finite, and $BV(u)$ is the set of the branches
 $\pi=(x_1,\dots, x_k, \dots)$ on which $u$ has finite variation
 \[
 \sum_{k=1}^{\infty}|u(x_{k+1})-u(x_k)|.
 \]
Clearly $BV(u)\subseteq\F(u).$

In \cite[Theorem A]{KLW}, the authors show that: If $F$ is an averaging
operator and $\mathcal{H}^m_F$ is the set of bounded $F$-harmonic
functions on $\T,$ then
\[
\min_{\mathcal{H}^m_F}\dim \F(u)=\min_{\mathcal{H}^m_F}\dim BV(u)
=\log_m(\tau(m,F)),
\]
where
$$\tau(m,F)=\min\left\{\displaystyle
\sum_{j=1}^m e^{x_j}\colon F(x_1,\dots,x_m)=0\right\}$$ and
$\dim$ denotes the usual Hausdorff dimension.

In \cite{KW}, for the classical $p$-harmonic
functions on trees (Example \ref{ex.p.harmonic}), the authors prove that
\[
\lim_{m\to+\infty} \min_{\mathcal{H}^m_{F^p}}{\dim}\, \F(u)=\lim_{m\to+\infty}\min_{\mathcal{H}^m_{F^p}}{\dim}\, BV(u)=1.
\]
While from \cite{dpmr}, for the $(\alpha,\beta)$-harmonious functions on trees (Example \ref{ex.p.harmonious}),
we have that
\[
\lim_{m\to+\infty} \min_{\mathcal{H}^m_{F_0}}{\dim}\,
\F(u)=\lim_{m\to+\infty}\min_{\mathcal{H}^m_{F_0}}{\dim}\, BV(u)=
\frac{1}2+\frac{\beta}2.
\]

In the case $F=F_1,$
we observe that the minimum  $\tau(m,F_1)$ is attained at
\[
x_i=\begin{cases}
-\frac{m+(1-s)(1-\alpha)}{m}
\log\gamma&\mbox{if } 1\le i\le s-1,\\
\frac{(s-1)(1-\alpha)}{m}\log\gamma &\mbox{if } s\le i \le m,
\end{cases}
\]
where
\[
\gamma=\frac{m+(1-s)(1-\alpha)}{(m-s+1)(1-\alpha)}\quad \mbox{ and } \quad	
s=
\begin{cases}
[\frac{m}{2}]+1& \mbox{ if } m \mbox{ is odd },\\
\frac{m}{2}&  \mbox{ if } m \mbox{ is even }.
\end{cases}
\]
Therefore
\[
\begin{array}{l}
\displaystyle \min_{\mathcal{H}_{F_1}}{\dim}\, \F(u)=\min_{\mathcal{H}_{F_1}}
{\dim}\, BV(u) \\[10pt]
\displaystyle \qquad = \log_m \left(\frac{m}{1-\alpha}
\left(\frac{(m-s+1)(1-\alpha)}{m+(1-s)(1-\alpha)}\right)^{1-\frac
{s-1}{m}(1-\alpha)}\right).
\end{array}
\]

Finally, in the case $F=F_2,$ the minimum $\tau(m,F_2)$ is attained at
\begin{align*}
x_{j}&=\frac{1-\alpha}{2}\log\eta,\, 1\le j\le m-1,
\, x_m=-\frac{1+\alpha}2
\log\eta,&\quad \mbox{if } 0\le\alpha\le\frac{m-2}m,\\
x_{1}&=-\frac{1+\alpha}2
\log\eta,\quad \quad x_j=\frac{1-\alpha}{2}\log\eta,\, 2\le j\le m,
&\quad \mbox{if } \frac{m-2}m<\alpha<1,  	
\end{align*}
where
\[
\eta=\frac{1+\alpha}{(m-1)(1-\alpha)}.
\]
Then
\[
\min_{\mathcal{H}_{F_2}}{\rm dim}\, \F(u)=\min_{\mathcal{H}_{F_2}}{\rm dim}\, BV(u)
=
\log_m\left(\frac{2(m-1)(1+\alpha)^{-\frac{1+\alpha}2}}
{\left((m-1)(1-\alpha)\right)^{\frac{1-\alpha}2}}\right).
\]

Thus, we can compute the following limits as the number of branches go to infinity,
\begin{align*}
&\lim_{m\to+\infty}\min_{\mathcal{H}_{F_1}}{\dim}\, \F(u)=\lim_{m
\to+\infty}\min_{\mathcal{H}_{F_1}}{\dim}\, BV(u)=1,\\
&\lim_{m\to+\infty}\min_{\mathcal{H}_{F_2}}
{\dim}\, \F(u)=\lim_{m\to
+\infty}\min_{\mathcal{H}_{F_2}}{\dim}\, BV(u)= \frac{1}2+\frac
{\alpha}2.
\end{align*}

\subsection{$F$-harmonic Measure}
Let $F$ be an averaging operator, $f:[0,1]\to\R$ be a bounded function,
$c>0$ and $E\subset\partial\T.$

We define $U_F(f,E,c)$ as the set of all $F-$superharmonic functions $u$ such that
\[
\liminf_{k\to+\infty} u(x_k)\ge
f(\pi)+c\CChi_{E}(\pi),
\]
for all $\pi=(x_1, \dots, x_k, \dots)\in\partial\T.$

When $f\equiv 0$ and $c=1,$ we say that $U_F(0,E,1)$ is the upper class of $E$, and
\[
\omega_{F}(x,E)\defi\inf\left\{u(x)\colon u\in U_F(0,E,1)\right\}
\]
is the $F$-harmonic measure function for $E$. We call
$\omega_F(E)\defi
\omega_F(\emptyset,E)$ the $F$-harmonic measure of $E.$

Let  $E$ be a subset of $\partial\T,$ following the arguments in \cite{HKM}, we have that
\begin{enumerate}[(a)]
\item $0\le \omega_F(\cdot, E)\le 1$ on $\T;$
\item  $ \omega_F(E)\le  \omega_F(G)$ when $E\subset G;$
\item  $\omega_F(\cdot, E)$ is $F-$harmonic on $\T;$
\item If $E$ is compact, then $\lim_{k\to+\infty}
\omega_F(x_k, E)=0$ for
$\pi= (x_1, \dots, x_k, \dots)\in \partial\T\setminus E;$
\item If $E$ is compact, then
$\omega_F(E)+\omega_F(\partial\T\setminus E)=1.$
\end{enumerate}

Moreover, if $E$ and $G$ are disjoint compact sets on
$\partial\T$ and $\omega_F(E)=\omega_F(G)=0,$ then
$\omega_F(E\cup G)=0.$ Lastly if
$E_1\supseteq E_2\supseteq
\cdots \supseteq E_j\supseteq\cdots$ are compact sets then
\[
\lim_{j\to+\infty} \omega_F(E_j)=
\omega_F\left(\bigcap_{j=1}^{\infty} E_j\right).
\]

In \cite{KLW}, the authors show that
$F$-harmonious measures on trees lack many desirable
properties of set valued functions find in classical analysis.
More precisely, if $F$ is a permutation invariant averaging
operator, not equal to the usual average, then
$\omega_F$ is not a Choquet capacity,
union of sets of $\omega_F$ measure zero can have positive
$\omega_F$ measure and there exist sets of full $\omega_F$ measure having small dimension. See also \cite{ary}.

\section{The Dirichlet Problem}\label{DP}
We now introduce what we understand by the Dirichlet problem
in this work.

\medskip

\noindent{\bf Dirichlet Problem} (DP). Given an averaging operator $F$
 and a bounded function $f\colon [0,1]\to\R,$ find a $F$-harmonic
 function $u$  such that
\[
\lim_{k\to+\infty} u(x_k)= f(\pi) \quad \forall \pi=(x_1, \dots, x_k, \dots)\in\partial\T.
\]

We say that $u$ is a supersolution of (DP) if $u$ is a
$F$-superharmonic function  and
\[
\liminf_{k\to+\infty} u(x_k)\ge f(\pi)
\quad \forall \pi=(x_1, \dots, x_k, \dots)\in\partial\T.
\]

We say that $u$ is a subsolution of (DP)
if $-u$ is a supersolution of (DP) with boundary data $-f.$

\subsection{Existence}
In this subsection, following \cite[Section 4]{dpmr}, we give a proof
of existence of solutions of (DP) when the boundary data is
a continuous function.

Let $f\colon[0,1]\to\R$ be a bounded function and $n\in\N,$ we
define $f_n:[0,1]\to\R$ as
\[
f_n(t)\defi\sum_{j=0}^{m^n-1}f\left(\nicefrac{j}{m^n}\right)\CChi_{I_{nj}}(t)
\]
where $I_{nj}\defi
\left[\nicefrac{j}{m^n},\nicefrac{(j+1)}{m^n}\right)$
for all $j\in\{0,\dots,m^n-1\}$ and
$I_{n(m^n-1)}\defi \left[\nicefrac{(m^n-1)}{m^n},1\right]$. Note that
this function is piecewise constant.

\medskip

Our next goal is to construct a $F$-harmonic function $u_n$
such that $u_n(x)=f_n(x)$
for all $x\in\T^k$ for any $k\ge n.$

We first observe that, for all
$j\in\{0,\dots,m^n-1\}$  there exists  $x_{nj}\in\T^n$ such that
$I_{x_{nj}}=\overline{I_{nj}}.$ Then, for all
$k\in\{1,\dots,n\}$, we take
$\{x_{(n-k)j}\}_{j=0}^{m^{n-k}-1}\subset\T$ such that
\[
\S(x_{(n-k)j})=\{x_{(n-k+1)\tau}\colon 1+(j-1)m\le\tau\le jm\} \quad\forall j\in\{0,\dots,m^{n-k}-1\}.
\]

Let $u_n\colon\T\to\R$ be such that
\begin{equation}\label{river.campeon}
 u_n(y)\defi f\left(\nicefrac{j}{m^n}\right),
\end{equation}
 for all $y\in\T^{x_{nj}}$  for some ${j\in\{1,\dots,m^n-1\}},$
 and
\[
u_n(x_{(n-k)j})\defi F(u_n(x_{(n-k)j},0),\dots,u_n(x_{(n-k)j},m-1)),
\]
for any $k\in\{1,\dots,n\}$ and for all $j\in\{0,\dots,m^{n-k}-1\}$.

It is easy to check that
$u_n$ is a $F$-harmonic function. Moreover,
 $\{u_n\}_{n\in\N}$ is uniformly bounded on
$\T$ due to the fact that $f$ is bounded.

\begin{re}\label{cu}
Let $f$
be a continuous function on $[0,1].$ Then, given $\varepsilon>0$
there exists $\delta=\delta(\varepsilon)>0$ such that
\[
|f(x)-f(y)|\le \frac{\varepsilon}{2}+\frac{2\|f\|_{\infty}}{\delta}|x-y|,
\]
for all $x,y\in[0,1].$ Therefore, for any $n\in\N$ and
$j\in\{0,\dots,m^n-1\}$ we have that
\[
|f_n(x)-f(y)|\le \frac{\varepsilon}2+
\frac{2\|f\|_{\infty}}{\delta m^n},
\]
for all $x, y\in I_{nj}.$ Then $\{f_n\}_{n\in\N}$ converges
uniformly to $f.$
\end{re}

We are now ready to state our existence result for the Dirichlet problem.

\begin{teo}\label{edp1}
Let $F$ be an averaging operator and
$f:[0,1]\to\R$ be a continuous function. Then the sequence
$\{u_n\}_{n\in\N}$ converges uniformly to a solution $u$ of
\textnormal{(DP)} with boundary data $f.$ Moreover,
if $f$ is a Lipschitz
function we have a bound for the error, it holds that
\[
|u_{n}(x)-u(x)|\le \frac{ L}{m^n},
\]
for all $x\in\T,$ where $L$ is the Lipschitz constant of $f.$
\end{teo}

\begin{proof} The proof is divided into 3 steps.

\noindent{\bf Step 1}. First, we prove that
$\{u_n\}_{n\in\N}$ is an uniformly Cauchy sequence.
Let $h,k,n\in\N$ and $x\in\T^h.$ If $n\le k\le h,$
there exist $i\in\{0,\dots,m^n-1\}$ and $j\in\{0,\dots,m^k-1\}$
such that  $u_n(x)=f_n(x)=f\left(\nicefrac{i}{m^n}\right)$ and
$u_k(x)=f_k(x)=f\left(\nicefrac{j}{m^k}\right).$ Moreover,
$I_x\subset I_{x_{kj}}\subset I_{x_{ni}}.$ Then, given
$\varepsilon>0,$ by Remark \ref{cu}, there exists
$\delta=\delta(\varepsilon)>0$ such that
\[
|u_n(x)-u_k(x)|\le\frac{\varepsilon}2
+ \frac{2\|f\|_{\infty}}{\delta m^n},
\quad\forall x\in\T^h.
\]
Thus, there exists $n_0$ such that if
$k\ge n\ge n_0,$
\[
|u_n(x)-u_k(x)|\le\varepsilon, \quad \forall x\in\T^h.
\]
Then for any $x\in\T^{h-1},$ by the above inequality,
\[
u_k(y)-\varepsilon\le u_n(y)\le u_k(y)+\varepsilon,\quad\forall
y\in\S(x).
\]
Therefore, since $u_n$ is a $F$-harmonic function and using (iii)
and (v), we have
\[
u_k(x)-\varepsilon\le u_n(x)\le u_k(x)+\varepsilon\quad\forall
x\in\T^{h-1},
\]
that is,
\[
|u_n(x)-u_k(x)|\le\varepsilon \quad\forall
x\in\T^{h-1}.
\] 	

In the same manner, in $(h-1)$-steps, we can see
that	
\[
|u_n(x)-u_k(x)|\le\varepsilon \quad\forall
x\in\T.
\] 	
Therefore $\{u_n\}_{n\in\N}$ is an uniformly Cauchy sequence.

\noindent{\bf Step 2}. Now we show that
\[
u(x)\defi \lim_{n\to+\infty} u_n(x) \quad \forall x\in\T
\]
is a solution of (DP) with boundary data $f.$
By step 1, $\{u_n\}_{n\in\N}$ converges uniformly to $u.$ Therefore,
by Lemma \ref{lim-unif}, $u$ is a $F$-harmonic function. Then,
we only need to show that
\[
\lim_{k\to +\infty}u(x_k)=f(\pi)\quad \forall
\pi=(x_1,\dots,x_k,\dots)\in\partial\T.
\]

Let  $\pi=(x_1,\dots,x_k,\dots)\in\partial\T$ and
$\varepsilon>0.$ Since
$\{u_n\}_{n\in\N}$ converges uniformly to $u,$ there exists
$n_0=n_0(\varepsilon)$ such that
\begin{equation}
		\label{desconunif1}|u_n(x_j)-u(x_j)|\le\frac{\varepsilon}2,
		\quad\forall j\in\N
\end{equation}
if $n\ge n_0.$

On the other hand, we can observe that there exists
$n_1=n_1(\varepsilon)$ such that
\[
|f_n(\pi)-f(\pi)|\le\frac{\varepsilon}2
\]
if $n\ge n_1.$ Then, since $u_n(x)=f_n(x)$ for all $x\in\T^h$
for any $h\ge n,$ if $n\ge n_1$ we have that
\begin{equation}
		\label{desconunif2}|u_n(x_j)-f(\pi)|
		\le\frac{\varepsilon}2 \quad\forall j\ge n.
\end{equation}

Finally, taking $n\ge\max\{n_0,n_1\}$ and $j\ge n,$ by
\eqref{desconunif1} and \eqref{desconunif2}, we get
\[
|u(x_j)-f(\pi)|\le|u(x_j)-u_n(x_j)|+|u_n(x_j)-f(\pi)|\le
\varepsilon.
\]

\noindent{\bf Step 3.}
We observe that if $f$ is a Lipschitz function, in the same manner
as in step~1, we obtain that, if $k,n\in\N,$
\[
|u_n(x)-u_k(x)|\le \frac{L}{m^n} \quad\forall x\in\T.
\]
Therefore,
\[
|u_n(x)-u(x)|\le \frac{L}{m^n} \quad\forall x\in\T,
\]
where $L$ is the Lipschitz constant of $f.$ This completes the proof.
\end{proof}

\subsection{Uniqueness} The comparison principle and the uniqueness
of our Dirichlet problem follow immediately from the following
lemma.

\begin{lem}\label{cotasupdp1}
	Let $F$ be an averaging operator and
	$f,g\colon [0,1]\to\R$ be bounded functions. If $u$ is
	a subsolution of \textnormal{(DP)} with boundary data $f$ and
	$v$ is a supersolution of \textnormal{(DP)} with boundary data
	$g$ then
	\[
	\sup_{x\in\T} \left\{u(x)-v(x)\right\}\le
	\sup_{x\in[0,1]} \left\{f(x)-g(x)\right\}.
	\]
\end{lem}

\begin{proof} Let $M=\sup_{x\in\T}
\left\{u(x)-v(x)\right\}.$ Then,  given $\varepsilon>0$ there exists
$x_{k_0}\in\T$ such that
\begin{align*}
	M-\varepsilon&\le u(x_{k_0})-v(x_{k_0})\\
	&\le F(u(x_{k_0},0),\dots,u(x_{k_0},m-1))-
	F(v(x_{k_0},0),\dots,v(x_{k_0},m-1)).
\end{align*}
Using Remark \ref{Fmax}, we get
\[
M-\varepsilon\le \max_{y\in\S(x_{k_0})}\left\{u(y)-v(y)\right\}.
\]

Thus, taking $x_{k_1}\in\S(x_{k_0})$ such that
$\displaystyle \max_{y\in\S(x_{k_0})}\left\{u(y)-v(y)\right\}=
u(x_{k_1})-v(x_{k_1}),$ we have that
\[
M-\varepsilon\le u(x_{k_1})-v(x_{k_1}).
\]
Continuing this reasoning, we obtain by induction that for all
$j \ge 1$
there exists $x_{k_j}\in\S(x_{k_{j-1}})$ such that
\begin{equation}\label{cotaeps}
	M-\varepsilon\le u(x_{k_j})-v(x_{k_j}).
\end{equation}

Now, since
\[
\limsup_{k\to+\infty}u(x_k)\le f(\pi) \quad\mbox{ and } \quad
\liminf_{k\to+\infty}v(x_k)\ge g(\pi)
\]
for all $\pi=(x_1,\dots,x_k,\dots)\in\partial\T,$
by \eqref{cotaeps}, we have that
\[
M-\varepsilon\le f(\pi_0)-g(\pi_0)
\]
where $\pi_0=(x_{k_0},\dots,x_{k_j},\dots).$ Therefore
\[
M-\varepsilon\le \sup_{x\in[0,1]}\left\{f(x)-g(x)\right\}.
\]
Since $\varepsilon$ is arbitrary, the proof is complete.
\end{proof}

The above lemma implies the comparison principle for solutions of
(DP).

\begin{teo}[Comparison Principle]\label{cpdp1}
Let $F$ be an averaging operator and
 $f,g\colon[0,1]\to\R$ be bounded functions. If $v$ is a
supersolution (resp. subsolution) of
\textnormal{(DP)} with boundary data $g$,
$u$ is a solution of
\textnormal{(DP)} with boundary data $f$ and $f\le g$
(resp. $f\ge g$), we have that $u\le v$ (resp. $u\ge v$).
\end{teo}

Now, we arrive to the main result of this section.

\begin{teo} Let $F$ be an averaging operator and
 $f\colon[0,1]\to\R$ be a bounded function.
	There exists a unique bounded solution of \textnormal{(DP)}
	 with boundary data $f.$
\end{teo}

\begin{proof} Theorem \ref{edp1}
gives a solution of (DP)
and the comparison principle implies the uniqueness.
\end{proof}

\begin{re}
	Observe that the sequence $\{u_n\}_{n\in\N}$ given by \eqref{river.campeon} converges uniformly
	to the unique
	solution of  \textnormal{(DP)}.
\end{re}

\section{$F$-harmonic Measure Estimates}\label{esti}
In this section we give some estimates for $F$-harmonic measures. First
we introduce some definitions.

\begin{de}\label{defiv}
Let $F$ be an averaging operator, $f:[0, 1]\to\R$ be a continuous
function, $u: \T\to\R$  the solution of \textnormal{(DP)} with boundary data $f,$
 $c>0$ and $n\in\mathbb{N}.$ Given
 $I={\bigcup_{j=k_0}^{k_1}}\left[\frac{j}{m^n},
 \frac{j+1}{m^n}\right]$
 with $0\le k_0\le k_1\le m^n-1,$ we define
 $\vv=\mathfrak{v}: \T\to\R$ by
\[
\mathfrak{v}(x)\!=\!
\begin{cases}
u(x)+c &\mbox{if } x\in K ,\\
u(x)&\mbox{if }x\in\T^y\mbox{ with } y\in\T^n\setminus \{x_j\}_{j=k_0}^{k_1},\\
F(\mathfrak{v}(x,0), \dots,\mathfrak{v}(x,m-1)) &\mbox{if } x\in\T^j, 0\le j\le n-1,	
\end{cases}
\]
where  $\{x_j\}_{j=k_0}^{k_1}$ is the unique set of
vertices of $\T^n$ such that
$I={ \bigcup_{j=k_0}^{k_1}} I_{x_j}$ and $K ={ \bigcup_{j=k_0}^{k_1}} \T^{x_j}.$
\end{de}

\begin{re}
The function $\vv$ is  $F$-harmonic.
\end{re}

We now prove some technical results.

\begin{lem}\label{vinf}
Let $F$ be a permutation invariant averaging operator
with the pro\mbox{per}ty \eqref{Pro}, $f:[0, 1]\to\R$ be a continuous
function and $c>0.$  If
 $I={
 \bigcup_{j=k_0}^{k_1}}\left[\frac{j}{m^n},  \frac{j+1}{m^n}
 \right]$
 with $0\le k_0\le k_1\le m^n-1$ and $k_0+k_1\neq m^n-1,$
then
\begin{equation}\label{infenv}
\vv(\emptyset)= \inf\left\{ w(\emptyset)\colon w\in U_F(f,I,c)\right\}.
\end{equation}
\end{lem}

\begin{proof}
Let $\{x_j\}_{j=k_0}^{k_1}\subset \T^n$ such that
\[
I=\bigcup_{j=k_0}^{k_1} I_{x_j}.
\]

Given $w\in U_F(f,I,c),$ we have that:
\begin{itemize}
\item For each $j\in\{k_0,k_0+1,\dots,k_1\},$ $\vv$ and $w$
are the solution and a supersolution of (DP) on $\T^{x_j}$
 with boundary data $f+c\Chi,$ respectively;
\item For any $z\in\T^n\setminus\{x_j\}_{j=k_0}^{k_1},$
$\vv$ and $w$ are the solution and a supersolution of (DP) on $\T^{z}$
 with boundary data $f,$ respectively.
\end{itemize}
 Thus, by the comparison principle,
\[
\vv(x)\le w(x), \quad \forall x\in\T^y
\]
for any $y\in\T^n.$
Therefore, using that $\vv$ is $F$-harmonic, $w$ is
$F$-superharmonic and (v),
we have that
\[
\vv(x)\le w(x),
\]
for all $x\in\T.$ In particular
\[
\vv(\emptyset)\le w(\emptyset).
\]

Since $w\in U_F(f,I,c)$ is arbitrary, we obtain that
\[
\vv(\emptyset)\le \inf\left\{ w(\emptyset)\colon w\in U_F(f,I,c)\right\}.
\]

To prove the opposite inequality, we will construct a sequence
$\{w_l\}_{l\in\N}\subset U_F(f,I,c)$ such that
\[
\vv(\emptyset)= \lim_{l\to+\infty}w_l(\emptyset).
\]

To this end, we need to study three cases.

\noindent{\bf Case 1.} First we study the case $k_0=0.$

Let $l\in\N.$ We define $I_l\defi
\left[\frac{k_{1}+1}{m^n}, \frac{k_{1}+1}{m^n}+\frac{1}{m^{n+l}}
\right]$ and
$w_l(x)\defi \mathfrak{v}_{(f,I\cup I_l,c)}(x).$
Observe that $w_l$ is a $F$-harmonic function such
that
\[
\liminf_{k\to+\infty}w_l(x_k)\ge f(\pi)+c\Chi(\pi)
\quad \forall \pi=(x_1, \dots, x_k, \dots)\in\partial
\T.
\]
Then $w_l\in U_F(f,I,c)$ for all $l\in\N.$
Moreover, by (v), $\{w_l\}_{l\in\N}$
is a nonincreasing sequence and
\[
w_l(x)\ge \vv(x)
\]
for all $x\in\T$ and $l\in\N.$

Finally, we will prove the following inequality
\begin{equation}\label{desvw1}
\vv(\emptyset)\le w_l(\emptyset)\le \vv(\emptyset)+c\kappa^{n+l}
\end{equation}
for all $l\in\N.$

Let $l\in\N$ and $z_1\in\T^{n+l-1}$ such that
$I_l=I_{z_0}$ where $z_0=(z_{1},0).$
Then, for any $x\in\T^{n+r}$ with $r\in\N_{l},$ we have that
\[
w_l(x)=
\begin{cases} \vv(x)& \mbox{if } x\in\T^y \mbox{ with } y\in\T^{n+l}\setminus \{z_0\},\\
\vv(x)+c & \mbox{if } x\in\T^{z_0}.
\end{cases}
\]
Then
\begin{equation}\label{des-wl1}
w_l(x)= \vv(x) \quad \forall x\in\T^{n+l-1}\setminus \{z_1\},
\end{equation}
and,  by (v), \eqref{Pro} and the fact that $\vv$ is $F$-harmonic, we get
\begin{equation}\label{des-wl2}
\begin{aligned}
w_l(z_1)&=F(w_l(z_1,0),w_l(z_1,1),\dots,w_l(z_1,m-1))\\
&=F(\vv(z_0)+c,\vv(z_1,1),\dots,\vv(z_1,m-1))\\
&\le F(\vv(z_1,0),\vv(z_1,1),\dots,\vv(z_1,m-1))+c\kappa\\
&=\vv(z_1)+c\kappa.
\end{aligned}
\end{equation}

Let $z_2\in\T^{n+l-2}$ such that $z_1=(z_2,0).$ Then, by \eqref{des-wl1},
\[
w_l(x)= \vv(x) \quad \forall x\in\T^{n+l-1}\setminus \{z_2\},
\]
and, using that $w_l$ is $F$-harmonic, \eqref{des-wl2}, (v), \eqref{Pro} and the fact that $\vv$ is $F$-harmonic, we get
\begin{align*}
w_l(z_2)&=F(w_l(z_2,0),w_l(z_2,1),\dots,w_l(z_2,m-1))\\
&=F(w_l(z_1),\vv(z_2,1),\dots,\vv(z_2,m-1))\\
&\le F(\vv(z_1)+c\kappa,\vv(z_2,1),\dots,\vv(z_2,m-1))\\
&\le  F(\vv(z_2,0),\vv(z_2,1),\dots,\vv(z_2,m-1))+c\kappa^2\\
&= \vv(z_2)+ c\kappa^2.
\end{align*}

By repeating this procedure $n+l-2$ times we can obtain \eqref{desvw1}. Therefore
taking limit as $l\to+\infty$ in \eqref{desvw1}, we have that
\[
\lim_{l\to+\infty}w_l(\emptyset)=\vv(\emptyset).
\]

\noindent{\bf Case 2}.  $k_1=m^n-1.$ The proof of this case is similar to the previous one.

\noindent{\bf Case 3}.
Finally we will study the case $0<k_0\le k_1<m^n-1.$

Let $l\in\N.$ We define
$I_l^1=[\frac{k_0}{m^n}-\frac1{m^{n+l}},\frac{k_0}{m^n}],$
$I_l^2=[\frac{k_{1}+1}{m^n},\frac{k_{1}+1}{m^n}+\frac1{m^{n+l}}]$ and
$w_l(x)\defi\mathfrak{v}_{(f,I_l^1\cup I \cup I_l^2,c)}.$
As in case 1,  $w_l\in U_F(f,I,c)$ for all $l\in\N$ and
\[
w_l(x)\ge \vv(x)
\]
for all $x\in\T$ and $l\in\N.$

We will prove the following inequality
\[
\vv(\emptyset)\le w_l(\emptyset)\le \vv(\emptyset)+2c\kappa^{n+l},
\]
for all $l\in\N.$

Let $l\in\N$ and $z_1^1, z_1^2 \in\T^{n+l-1}$ such that
$I_l^1=I_{z_0^1}$ and $I_j^2=I_{z_0^2}$ where
$z_0^1=(z_1^1,m^n-1)$ and  $z_0^2=(z_1^2,0).$
Observe that
$z_1^1\neq z_1^2.$

Then, for any $x\in\T^{n+r}$ with $r\in\N_{l},$ we have that
\[
w_l(x)=
\begin{cases} \vv(x)& \mbox{if } x\in\T^y \mbox{ with } y\in\T^{n+l}\setminus \{z_0^1,z_0^2\},\\
\vv(x)+c & \mbox{if } x\in\T^{z_0^1}\cup\T^{z_0^2}.
\end{cases}
\]
Then
\begin{equation}\label{des-wl11}
w_l(x)= \vv(x) \quad \forall x\in\T^{n+l-1}\setminus
\{z_1^1,z_1^2\}
\end{equation}
and,by (v), \eqref{Pro}, using that
$F$ is a permutation invariant averaging operator
and the fact that $\vv$ is $F$-harmonic, we get
\begin{equation}\label{des-wl22}
\begin{aligned}
w_l(z_1^j)&=F(w_l(z_1^1,0),w_l(z_1^1,1),\dots,w_l(z_1^1,m-1))\\
&=F(\vv(z_1^1,0),\vv(z_1^1,1),\dots,\vv(z_0^1)+c)\\
&\le F(\vv(z_1^1,0),\vv(z_1^1,1),\dots,\vv(z_1^1,m-1)))+c\kappa\\
&=\vv(z_1^1)+c\kappa
\end{aligned}
\end{equation}
and
\[
\begin{aligned}
w_l(z_1^2)&=F(w_l(z_1^2,0),w_l(z_1^1,1),\dots,w_l(z_1^2,m-1))\\
&=F(\vv(z_0^2)+c,\vv(z_1^2,1),\dots,\vv(z_1^2,m-1))\\
&\le F(\vv(z_1^2,0),\vv(z_1^2,1),\dots,\vv(z_1^2,m-1))+c\kappa\\
&=\vv(z_1^2)+c\kappa.
\end{aligned}
\]

Let $z_2^1,z_2^2\in\T^{n+l-2}$ such that $z_1^1=(z_2^1,m-1)$
and $z_1^2=(z_2^2,0).$ Then, by \eqref{des-wl11},
\[
w_l(x)= \vv(x) \quad \forall x\in\T^{n+l-1}\setminus
\{z_2^1,z_2^2\}.
\]
Using that $w_l$ is $F$-harmonic,
\eqref{des-wl22}, (v), \eqref{Pro},
$F$ is a permutation invariant averaging operator
and the fact that $\vv$ is $F$-harmonic, we get
\begin{align*}
w_l(z_2^1)&=F(w_l(z_2^1,0),w_l(z_2^1,1),\dots,w_l(z_2^1,m-1))\\
&=F(\vv(z_2^1,0),\vv(z_2^1,1),\dots,w_l(z_1^1))\\
&\le F(\vv(z_2^1,0),\vv(z_2^1,1),\dots,\vv(z_1^1)+c)\\
&\le  F(\vv(z_2^1,0),\vv(z_2^1,1),\dots,\vv(z_2^1,m-1))+c\kappa^2\\
&= \vv(z_2^1)+ c\kappa^2
\end{align*}
and, by \eqref{des-wl22},
\begin{align*}
w_l(z_2^2)&=F(w_l(z_2^2,0),w_l(z_2^2,1),\dots,w_l(z_2^2,m-1))\\
&=F(\vv(z_1^2),\vv(z_2^2,1),\dots,w_l(z_2^2,m-1))\\
&\le F(\vv(z_1^2)+c,\vv(z_2^2,1),\dots,\vv(z_2^2,m-1)\\
&\le  F(\vv(z_2^2,0),\vv(z_2^2,1),\dots,\vv(z_2^2,m-1))+c\kappa^2\\
&= \vv(z_2^2)+ c\kappa^2.
\end{align*}

There exists $z\in\T^k,$ $0\le k\le n+l-1,$
such that $z_2^1, z_2^2\in\T^{z}$
and there exists $\{z_{j}^i\}_{j=3}^{n+l-1},$ $i=1,2,$ such that
\[
z_{j}^i\in \S(z_{j+1}^i)\quad\forall j\in\{2,\dots,n+l-k-1\}
\quad\forall i\in\{1,2\},
\]
\[
z_{n+l-k}^1=z_{n+l-k}^2=z \mbox{ and }
 z_{n+l-k-1}^1\neq z_{n+l-k-1}^1.
\]

Arguing as before, for any $j\in\{0,\dots,n+l-k-1\}$
we have that
\begin{align*}
w_l(x)&= \vv(x), \quad \forall x\in\T^{n+l-j}\setminus
\{z_j^1,z_j^2\},\\
w_l(z_j^i)&\le\vv(z_j^i)+c\kappa^j, \quad i\in\{1,2\}.
\end{align*}

Therefore
\[
w_l(x)= \vv(x), \quad \forall x\in\T^{k}\setminus
\{z\},\\
\]
and using that $w_l$ is $F$-harmonic, (v), Remark \ref{doblePro},
that $F$ is a permutation invariant averaging operator and
the fact that $\vv$ is $F$-harmonic, we get
\[
w_l(z)\le \vv(z)+2c\kappa^{n+l-k}.
\]

Then, the following inequality
\[
\vv(\emptyset)\le w_l(\emptyset)\le\vv(\emptyset)+2c\kappa^{n+l}, \forall l\in\N
\]
can be proved in the same way as in the case 1. Therefore
\[
\lim_{l\to+\infty}w_l(\emptyset)=\vv(\emptyset).
\]
The proof is now complete.
\end{proof}

\begin{lem}\label{1int}
Let $F$ be a permutation invariant averaging operator
with the pro\mbox{per}ty \eqref{Pro}, $f:[0, 1]\to\R$ be a continuous
function,
$I= \left[\frac{k}{m^n}, \frac{k+1}{m^n}\right]$
with $k\in \{0, \dots, m^{n}-1\}$ ($n\in\mathbb{N}$)
and $c>0.$
If $u$ is the solution of
\textnormal{(DP)} with boundary
data $f$, then
\begin{equation}
\label{desbase}
0\le \vv(\emptyset)-u(\emptyset)\le c|I|^{\gamma}
\end{equation}
for all $\gamma\le -\log_m(\kappa).$

\end{lem}

\begin{proof} In  a similar way to the proof of Lemma \ref{vinf} (case 1),
we can prove that
\begin{equation}
\label{des1}
u(\emptyset)\le \vv(\emptyset)
\le u(\emptyset)+c\kappa^n.
\end{equation}

On the other hand, it is easy to check that
\begin{equation}\label{deskappa}
\kappa^n \le \left(\frac1{m^n}\right)^\gamma
\end{equation}
for all $\gamma\le -\log_m(\kappa).$ Therefore, by \eqref{des1}, \eqref{deskappa} and using that
$c>0,$ the inequality \eqref{desbase} holds.
\end{proof}

Now we consider the case where $I= \left[\frac{k}
{m^n}, \frac{k+1}{m^n}\right]
\cup\left[\frac{k+1}{m^n},
\frac{k+2}{m^n}\right]$ with
$k\in\{0,1,\dots,m^n-2\}$ ($n\in\mathbb{N}$).

\begin{lem}\label{2int}
	Let $F$ be a permutation invariant averaging operator
with the pro\mbox{per}ty \eqref{Pro}, $f:[0, 1]\to\R$ be a
continuous function, $I= \left[\frac{k}
{m^n}, \frac{k+1}{m^n}\right]
\cup\left[\frac{k+1}{m^n},
\frac{k+2}{m^n}\right]$ with
$k\in\{0,1,\dots,m^n-2\}$ ($n\in\mathbb{N}$) and $c>0.$
	If $u$ is the solution of \textnormal{(DP)}
	with boundary
	data $f,$  then	
	\[
0\le \vv(\emptyset)-u(\emptyset)\le
	2^{1-\gamma}c
	|I|^{\gamma}
         \]
	for all $\gamma\le -\log_m(\kappa).$
\end{lem}

\begin{proof}
In  a similar way to the proof of Lemma \ref{vinf}
(case 3), we can show that
\begin{equation}
	0\le \vv(\emptyset)
	\le u(\emptyset)+ 2c
	\kappa^{n}.
	\label{des3}
\end{equation}

Then, by \eqref{deskappa} and \eqref{des3}, we have that
\[
0\le \vv(\emptyset)-u(\emptyset)\le
	2c\left(\frac1{m^n}\right)^{\gamma}=
	2^{1-\gamma}c
	\left(\frac2{m^n}\right)^{\gamma}
\]
for all $\gamma\le -\log_m(\kappa).$
Thus
\[
0\le \vv(\emptyset)-u(\emptyset)\le
	2^{1-\gamma}c
	|I|^{\gamma}
\]
for all $\gamma\le -\log_m(\kappa).$
\end{proof}

Now,  we are able to prove Theorem \ref{teoestimacion}.
%

\begin{proof}[Proof of Theorem \ref{teoestimacion}]
We begin by taking
	\[
	n=\min\left\{l\in\mathbb{N}\colon \exists
	k\in\{0,\dots,m^{l}-1\}\mbox{ such that
	}\left[\frac{k}{m^l},\frac{k+1}{m^l}\right]\subset
	I\right\}.
	\]
	Observe that
	\begin{equation}
		\frac{1}{m^n}\le|I|<\frac{2}{m^{n-1}}
		\label{des5}
	\end{equation}
	and there exists
	$k_{n-1}\in\{0,\dots,m^{n-1}-2\}$
	such that
	\[
	I \subsetneq
	J_{n-1}\defi
	\left[\frac{k_{n-1}}{m^{n-1}},
	\frac{k_{n-1}+1}{m^{n-1}}\right]
	\cup
	\left[\frac{k_{n-1}+1}{m^{n-1}},
	\frac{k_{n-1}+2}{m^{n-1}}\right].
	\]
	Then $U_F(f,J_{n-1},c)\subset U_F(f,I,c)$ and
	therefore
	\[
	 \inf\left\{w(\emptyset)-u(\emptyset)
		\colon w\in U_F(f,I,c)\right\}
		\le\inf\left\{w(\emptyset)-u(\emptyset)
		\colon w\in U_F(f,J_{n-1},c)\right\}.
	\]
	Then, by Lemma \ref{vinf} and
	Theorem \ref{cpdp1}, we
	have that
	\begin{align*}
	\inf\left\{w(\emptyset)-u(\emptyset)
		\colon w\in U_F(f,I,c)\right\}
		&\le\inf\left\{w(\emptyset)-u(\emptyset)
		\colon w\in U_F(f,J_{n-1},c)\right\}\\
		&=\mathfrak{v}_{(f,J_{n-1},c)}(\emptyset)
		- u(\emptyset)\\
		&\le 2^{1-\gamma}c|J_{n-1}|^\gamma,
    \end{align*}
    for all $\gamma\le -\log_m(\kappa).$
	Therefore, using the fact that
	\[
         |J_{n-1}|=\frac{2m}{m^{n}}
	\]
	and
	\eqref{des5}, we
	have that
	\[
	0\le\inf\left\{w(\emptyset)-u(\emptyset)
		\colon w\in U_F(f,I,c)\right\}\le	
	2c (m|I|)^\gamma,
	\]
	for all $\gamma\le -\log_m(\kappa).$
\end{proof}

If $F$ is a permutation invariant averaging operator with the
property that there exists $0<\eta<1$ such that
\begin{equation}\label{Proinf}
F(x_1+c,x_2,\dots,x_m)\ge F(x_1,x_2,\dots,x_m)+c\eta
\end{equation}
for all $(x_1,\dots,x_m)\in\R^m$ and for all $c>0,$ arguing as
in Theorem \ref{teoestimacion}, we can show
the following result.

\begin{teo}
	Let $F$ be a permutation invariant averaging operator with
	the pro\mbox{per}ty \eqref{Proinf}, $f:[0, 1]\to\R$ be a continuous
	function, $I$ be a subinterval of $[0,1]$ and $c>0.$
	If $u$ is the solution of \textnormal{(DP)}
	with boundary
	data $f$, then
	\[
		\inf\left\{w(\emptyset)-u(\emptyset)
		\colon w\in U_F(f,I,c)\right\}\ge
		c\left(\frac{|I|}{2m}\right)^\theta
	\]
	for all $\theta\ge -\log_m(\eta).$
\end{teo}

\begin{ex}
The permutation invariant
averaging operator $F_0$ and $F_1$ satisfy
\eqref{Proinf} with $\eta=\frac{\beta}m.$	
\end{ex}

Finally, we prove Corollary \ref{BCP}.

\begin{proof}[Proof of Corollary \ref{BCP}]
We begin by observing that $
g\le f+ M\CChi_I$
due to $\|f\|_\infty+\|g\|_\infty\le M.$
Then, by Theorem \ref{cpdp1},
$
v(\emptyset)- u(\emptyset)\le w(\emptyset) - u(\emptyset)
$
for all $w\in U_F(f,I,M).$ Then
$
v(\emptyset)- u(\emptyset)\le
\inf\{w(\emptyset) - u(\emptyset)\colon w\in  U_F(f,I,M)\}.
$
Therefore, using Theorem~\ref{teoestimacion},
\[
v(\emptyset)- u(\emptyset)\le 2M(m|I|)^{\gamma}
\]
for all $\gamma\le -\log_m(\kappa).$
By a similar argument, we have that
\[
u(\emptyset)- v(\emptyset)\le 2M(m|I|)^{\gamma},
\]
for all $\gamma\le -\log_m(\kappa).$
Thus,
\[
|v(\emptyset)- u(\emptyset)|\le 2M(m|I|)^{\gamma}
=  2M(m\delta)^{\gamma},\quad \forall \gamma\le
-\log_m(\kappa).
\]

Finally, taking $\delta<\frac{1}m\left(\frac{\varepsilon}{2M}\right)^{\frac{1}{\gamma}},$ we get
$
|v(\emptyset)- u(\emptyset)|<\varepsilon,
$
which completes the proof.
\end{proof}


\end{document}